\documentclass[a4paper,reqno]{amsart}

\usepackage{amsmath,amssymb,enumerate}
\usepackage{hyperref}
\usepackage[all]{xy}

\newcommand{\diag}{\textup{diag}}
\newcommand{\tr}{\textup{tr}}
\newcommand{\ev}{\textup{ev}}
\newcommand{\id}{\textup{id}}
\newcommand{\aff}{\textup{Aff}}
\newcommand{\saff}{\textup{SAff}}
\newcommand{\ad}{\textup{Ad}}
\newcommand{\inv}{\textup{Ell}}
\newcommand{\cu}{\textup{Cu}}

\theoremstyle {plain}

\newtheorem {theorem}{Theorem}[section]
\newtheorem {lemma}[theorem]{Lemma}
\newtheorem {proposition}[theorem]{Proposition}
\newtheorem {corollary}[theorem]{Corollary}
\newtheorem {conjecture}[theorem]{Conjecture}

\theoremstyle {definition}

\newtheorem {definition}[theorem]{Definition}
\newtheorem {remark}[theorem]{Remark}

\numberwithin{equation}{section}

\begin{document}

\author{Bhishan Jacelon}
\address{Mathematisches Institut\\Einsteinstr. 62\\48149 M\"unster\\Germany} 
\email{b.jacelon@uni-muenster.de}
\date{\today}
\title{A simple, monotracial, stably projectionless $C^*$-algebra}
\thanks{Research supported by the ERC through AdG 267079.}

\maketitle

\begin{abstract}
We construct a simple, nuclear, stably projectionless $C^*$-algebra $W$ which has trivial $K$-theory and a unique tracial state, and we investigate the extent to which $W$ might fit into the hierarchy of strongly self-absorbing $C^*$-algebras as an analogue of the Cuntz algebra $\mathcal{O}_2$. In this context, we show that every nondegenerate endomorphism of $W$ is approximately inner and we construct a trace-preserving embedding of $W$ into the central sequences algebra $M(W)_\infty \cap W'$. We conjecture that $W\otimes W \cong W$ and we note some implications of this, for example that $W$ would be absorbed tensorially by a certain class of nuclear, stably projectionless $C^*$-algebras. Finally, we explain why $W$ may play some role in the classification of such algebras.
\end{abstract}
\maketitle

\section{Introduction} \label{introduction}

In the study of operator algebras, classification has always been a central theme. The classification of nuclear $C^*$-algebras was initiated by Elliott, who used ordered $K$-theory to classify approximately finite dimensional (AF) $C^*$-algebras in \cite{Elliott:1976kq} (building on earlier work of Glimm \cite{Glimm:1960qh} and Bratteli \cite{Bratteli:1972fj}) and approximately circle (A$\mathbb{T}$) algebras of real rank zero in \cite{Elliott:1993kq}. Around 1990, Elliott conjectured that the class of all separable, nuclear $C^*$-algebras could be classified by an invariant $\inv(\cdot)$ based on $K$-theory. Augmenting the invariant to include tracial data (and the natural pairing between traces and $K$-theory), this has come to be known as the Elliott Conjecture, and the project to establish its veracity, the Elliott Programme. The survey article \cite{Elliott:2008qy} gives an excellent overview of the history and recent developments of the Programme, and a more detailed exposition can be found in R\o rdam's monograph \cite{Rordam:2002yu}.

Today, it is known that the Elliott Conjecture does not hold in full generality. The ideas of \cite{Villadsen:1998ys} were used by R\o rdam \cite{Rordam:2003rz} and Toms \cite{Toms:2008hl} to construct simple, separable, nuclear non-isomorphic $C^*$-algebras with the same Elliott invariant. These counterexamples to the Conjecture can be dealt with in two ways: (1) enlarge the invariant to include the Cuntz semigroup (which is sensitive enough to be able to distinguish between the counterexamples of R\o rdam and Toms, and, for well-behaved $C^*$-algebras, can be recovered functorially from $K$-theory and traces---see \cite{Brown:2008mz}), or (2) impose further regularity conditions on the $C^*$-algebras to be classified. This article is in the spirit of option (2), and the relevant regularity property is $\mathcal{Z}$-stability.

The Jiang-Su algebra $\mathcal{Z}$ is a simple, separable, nuclear, infinite dimensional, projectionless $C^*$-algebra which has the same Elliott invariant as $\mathbb{C}$ (see \cite{Jiang:1999hb} and also \cite{Rordam:2009qy} for some alternative descriptions of $\mathcal{Z}$). A $C^*$-algebra $A$ is `$\mathcal{Z}$-stable' if $A\otimes \mathcal{Z} \cong A$; the counterexamples of R\o rdam and Toms each involve a pair of non-isomorphic $C^*$-algebras which have the same Elliott invariant, but one of which is not $\mathcal{Z}$-stable. Thus the class of $\mathcal{Z}$-stable $C^*$-algebras is the largest class in which the Elliott Conjecture can be expected to hold.

The term `strongly self-absorbing' was coined by Toms and Winter in \cite{Toms:2007uq} to describe a handful of algebras which, like $\mathcal{Z}$, have played pivotal roles in the classification programme. Specifically, a separable unital $C^*$-algebra $\mathcal{D} \ne \mathbb{C}$ is \emph{strongly self-absorbing} if there is an isomorphism $\varphi: \mathcal{D} \rightarrow \mathcal{D} \otimes \mathcal{D}$ such that $\varphi$ is approximately unitarily equivalent to the first factor embedding $\id_\mathcal{D}\otimes 1_\mathcal{D}$ (written $\varphi \sim_{a.u.} \id_\mathcal{D} \otimes 1_\mathcal{D}$), which means that there is  a sequence of unitaries $(u_n)_{n=1}^\infty$ in $\mathcal{D} \otimes \mathcal{D}$ such that $\varphi(a) = \lim_{n\rightarrow \infty} u_n(a\otimes 1)u_n^*$ for every $a\in \mathcal{D}$. As for $\mathcal{Z}$, a $C^*$-algebra $A$ is `$\mathcal{D}$-stable' if $A\otimes \mathcal{D} \cong A$.

If $A$ is strongly self-absorbing then $A$ is simple and nuclear, and is either purely infinite or stably finite with a unique tracial state. Moreover, the list of known strongly self-absorbing algebras is short (and closed under $\otimes$): the Cuntz algebras $\mathcal{O}_2$ and $\mathcal{O}_\infty$, UHF algebras of infinite type $M_{p^\infty}$ (such as the CAR algebra $M_{2^\infty}$), $\mathcal{O}_\infty \otimes M_{p^\infty}$, and the Jiang-Su algebra $\mathcal{Z}$. One of the reasons that these algebras are important is that they provide localized versions of the Elliott Conjecture in the sense of \cite{Winter:2007qf}, i.e.\ theorems of the form $\inv(A) \cong \inv(B) \implies A \otimes \mathcal{D} \cong B \otimes \mathcal{D}$. For example, the classification of Kirchberg algebras (\cite{Kirchberg:2009vn}, \cite{Phillips:2000fj}) can be interpreted as the classification up to $\mathcal{O}_\infty$-stability of all simple, separable, unital, nuclear $C^*$-algebras that satisfy the UCT. At the other end of the spectrum, Winter proves in \cite{Winter:2007qf} and \cite{Winter:2010hl} that the Elliott Conjecture holds for a large class of stably finite $\mathcal{Z}$-stable $C^*$-algebras (which have finite decomposition rank, as defined in \cite{Kirchberg:2004qy}).

The known strongly self-absorbing algebras form a hierarchy, with $\mathcal{Z}$ at the bottom (every strongly self-absorbing algebra is $\mathcal{Z}$-stable---see \cite{Winter:2009yq}) and $\mathcal{O}_2$ at the top ($\mathcal{O}_2$ absorbs every strongly self-absorbing algebra---see \cite{Kirchberg:2000kq}). Every purely infinite algebra in this hierarchy has a stably finite analogue (most notably, $\mathcal{Z}$ corresponds to $\mathcal{O}_\infty$), except for $\mathcal{O}_2$. That is, there is no stably finite strongly self-absorbing $C^*$-algebra $A$ with $K_*(A)=0$. In fact, it is not hard to see that if $A$ is  separable and stably finite  with $K_0(A) = 0$ then $A \otimes \mathcal{K}$ cannot have any full projections; in particular, $A$ must be nonunital (in fact, stably projectionless if $A$ is also simple). This explains the gap since the definition of `strongly self-absorbing'  conspicuously involves a unit. What is therefore needed to fill this gap is a well-behaved notion of `strongly self-absorbing' that makes sense for nonunital $C^*$-algebras and agrees with the existing definition for unital algebras.

There are a few equivalent characterizations of strongly self-absorbing $C^*$-algebras, some more amenable than others to a sensible interpretation in the nonunital case. The above definition is obviously troublesome and highlights a general problem: If $A$ is a nonunital $C^*$-algebra without projections then there are no obvious $^*$-homomorphisms from $A$ to $A\otimes A$; in particular, there is no obvious way of making sense of  an infinite tensor product $A^{\otimes \infty}$. On the other hand, if $A \ne \mathbb{C}$ is separable and unital, then $A$ is strongly self-absorbing if and only if (with some redundancy)

\begin{enumerate}[(i)]
\item $A \cong A \otimes A$;
\item every unital endomorphism of $A$ is approximately inner;
\item $A$ admits an asymptotically central sequence of unital endomorphisms (see section \ref{W is approx divisible}).
\end{enumerate}

These conditions also make sense if $A$ is nonunital (provided that we replace `unital' by `nondegenerate' where appropriate) and we could think of taking some subset of these as a general definition of strongly self-absorbing. It is not clear how well-behaved such a definition would be, but the goal of this paper is to find a stably projectionless $C^*$-algebra $W$ with trivial $K$-theory and a unique tracial state, and at least see how far these properties hold for $W$. The hope is that $W$ will play a role in the classification of stably projectionless $C^*$-algebras, similar to the roles played by $\mathcal{O}_\infty$ and $\mathcal{Z}$ in the classification of purely infinite and stably finite algebras respectively. 

Indeed, $W$ itself arises from a class of stably projectionless algebras which have been completely classified (by Razak \cite[Theorem 1.1]{Razak:2002kq} and Tsang \cite[Theorem 5.1]{Tsang:2005fj}). These algebras are inductive limits of building blocks of the form
\[
A(n,n') := \{ f\in C([0,1], M_{n'}) : f(0) = \diag(\overbrace{c,\ldots,c}^a,0_n) \: \text{and} \: f(1) = \diag(\overbrace{c,\ldots, c}^{a+1}),  \: c\in M_n \},
\]
where $n$ and $n'$ are natural numbers with $n|n'$ and  $a :=\frac{n'}{n}-1>0$. Each building block is stably projectionless and has trivial $K$-theory, so the Elliott invariant is purely tracial.

\begin{theorem}[Razak] \label{class}
Let $A$ and $B$ be simple inductive limits of (countably many) building blocks. If $(T^+A, \Sigma_A)$ is isomorphic to $(T^+B, \Sigma_B)$ then $A$ is isomorphic to $B$.
\end{theorem}

Here, $T^+A$ is the cone of densely defined lower semicontinuous traces on $A$ and $\Sigma_A$ is the compact convex subset of $T^+A$ consisting of those (bounded) traces of norm $\le 1$. We will call an inductive limit of countably many building blocks a \emph{Razak algebra}. (We may always assume that the connecting maps are injective---see section \ref{building blocks} below.)

Theorem \ref{class} has been superseded by the main result of \cite{Robert:2010qy}. There, a functor based on the Cuntz semigroup is used to classify (not necessarily simple) inductive limits of one-dimensional NCCW-complexes with trivial $K_1$-groups. In the case of simple $C^*$-algebras with trivial $K$-theory, this reduces to the following (see \cite[Corollary 6.2.4]{Robert:2010qy}).

\begin{theorem}[Robert] \label{class2}
Let $A$ and $B$ be inductive limits of one-dimensional NCCW-complexes which have trivial $K_1$-groups, such that $A$ and $B$ are simple, nonunital and have trivial $K$-theory. If $(T^+A, \Sigma_A)$ is isomorphic to $(T^+B, \Sigma_B)$ then $A$ is isomorphic to $B$.
\end{theorem}

For the range of the invariant, we still refer to \cite{Tsang:2005fj} (although an alternative proof will be provided in section \ref{W is self-absorbing} of this article).

\begin{theorem}[Tsang] \label{range}
Let $C$ be a topological convex cone that has a metrizable Choquet simplex as a base and let  $\omega$ be a faithful lower semicontinuous linear map $C \rightarrow [0,\infty]$. Then there exists a simple Razak algebra $A$ such that $(T^+A, \Sigma_A) = (C, \omega^{-1}([0,1]))$.
\end{theorem}

Suppose that $C=\mathbb{R}_+$ and let $\omega(x) = x$. By Theorems \ref{range} and \ref{class}, there exists a unique simple Razak algebra $W$ which has a tracial state $\tau$ such that $T^+W = \mathbb{R}_+ \tau$. That is, $W$ has a unique tracial state and every trace on $W$ is bounded. Note that $W \otimes \mathcal{K}$ is also a simple Razak algebra whose trace is unique up to scalar multiples, except that in this case the trace is unbounded. (This corresponds to taking  $\omega(x)= 0$ if $x=0$ and $\omega(x) = \infty$ otherwise.) Again, Theorem \ref{class} says that $W\otimes \mathcal{K}$ is the unique simple Razak algebra with this property.

It is easy to find a simple, nuclear,  nonunital $C^*$-algebra with a (unique) tracial state (for example, any proper hereditary subalgebra $A$ of $\mathcal{Z}$ has a tracial state, which is unique since $A$ is stably isomorphic to $\mathcal{Z}$), but to the author's knowledge, $W$ may be the first example of such an algebra which is stably projectionless. In particular, $W$ lives outside of the realm of $C^*$-algebras for which projections separate traces.

Note that Theorem \ref{class} also shows that $W$ absorbs the universal UHF algebra $\mathcal{Q}$, which implies that $W$ absorbs $\mathcal{Z}$ (see also section \ref{W is approx divisible}), and the Kirchberg-Phillips classification theorem (\cite{Kirchberg:2009vn}, \cite{Phillips:2000fj}) shows that $W\otimes A \cong \mathcal{O}_2 \otimes \mathcal{K}$ for any simple, separable, nuclear, stably infinite $C^*$-algebra $A$ that satisfies the UCT (see Theorem 8.4.1 and also Theorems 4.1.10 and 4.1.3 of \cite{Rordam:2002yu}). In particular, $W\otimes \mathcal{O}_2 \cong W \otimes \mathcal{O}_\infty \cong \mathcal{O}_2 \otimes \mathcal{K}$.

Finally, it should be noted that \cite{Robert:2010qy} shows that $W$ also arises from certain crossed products of $\mathcal{O}_2$ by $\mathbb{R}$. For $\lambda \in \mathbb{R}$, consider the action $\alpha$ of $\mathbb{R}$ on $\mathcal{O}_2  = C^*(s_1, s_2)$ given by
\[
\alpha_t(s_1) = e^{it}s_1, \qquad \alpha_t(s_2) = e^{\lambda it}s_2.
\]
These actions, and the associated crossed products $\mathcal{O}_2 \rtimes_\lambda  \mathbb{R}$, have been studied by many authors including Evans, Kishimoto and Kumjian (see \cite{Evans:1980kq}, \cite{Kishimoto:1996yu} and \cite{Kishimoto:1997kq}). Kishimoto and Kumjian proved that $\mathcal{O}_2 \rtimes_\lambda  \mathbb{R}$ is simple if and only if $\lambda$ is irrational; in this case $\mathcal{O}_2 \rtimes_\lambda  \mathbb{R}$ is stable, and is purely infinite if $\lambda<0$ and projectionless with a unique (unbounded) trace if $\lambda>0$. Moreover, Dean proves in \cite{Dean:2001qy} that for generic positive irrational $\lambda$, $\mathcal{O}_2 \rtimes_\lambda  \mathbb{R}$ can be written as a countable inductive limit of certain one-dimensional subhomogeneous $C^*$-algebras called NCCW complexes (see sections \ref{building blocks} and \ref{endomorphisms of W} below), which are shown in \cite{Robert:2010qy} to have trivial $K_1$-groups. Since, by \cite[Theorem IV.2]{Connes:1981qf}, these crossed products have trivial $K$-theory, we can apply the Kirchberg-Phillips classification theorem in the former case, and the classification theorem of \cite{Robert:2010qy} in the latter, to deduce that
\[
\mathcal{O}_2 \rtimes_\lambda  \mathbb{R} \cong \left\{
\begin{array}{lll}
\mathcal{O}_2 \otimes \mathcal{K} & \text{for every} & \lambda\in\mathbb{R}_-\backslash \mathbb{Q}\\
W\otimes \mathcal{K} & \text{for generic} & \lambda\in\mathbb{R}_+\backslash \mathbb{Q}.
\end{array}
\right.
\]

The paper is organized as follows. We first recall some basic facts about building blocks and establish notation in section \ref{building blocks}. In section \ref{construction of W} we will explicitly exhibit $W$ in a manner analogous to the construction of the Jiang-Su algebra $\mathcal{Z}$. Next, in section \ref{endomorphisms of W} we characterize $W$ as the unique terminal object in a certain category, from which it will follow that every nondegenerate endomorphism of $W$ is approximately inner. In section \ref{W is self-absorbing}, we make some remarks on the conjecture that $W\otimes W \cong W$ (an incorrect proof of which was offered in an earlier version of this article). We also prove here that, if this conjecture were to be true (and the Elliott Conjecture certainly predicts that it is), then among the $C^*$-algebras classified in \cite{Robert:2010qy}, those that are simple and have trivial $K$-theory would all absorb $W$ tensorially. Section \ref{W is approx divisible} contains a complete proof of \cite[Proposition 4.1]{Toms:2005kq}, which says that every simple Razak algebra is approximately divisible, and we show how the proof can be adapted to construct a trace-preserving embedding of $W$ into the central sequences algebra $M(W)_\infty \cap W'$ (where $M(W)$ is the multiplier algebra of $W$).

\subsection*{Acknowledgements} This work was part of my doctoral research at the University of Glasgow. I am indebted to: my supervisor Simon Wassermann, and also Stuart White, Rob Archbold and Ulrich Kraehmer, for carefully reading earlier versions  of the manuscript; Wilhelm Winter for suggesting the topic and for many illuminating discussions; and Andrew Toms and George Elliott for their very helpful correspondence. I am also enormously grateful to Luis Santiago and Leonel Robert for helping to uncover the error in the attempted proof of Conjecture \ref{absorbing}.

\section{The building blocks} \label{building blocks}

Throughout this section, let $A$ be the building block
\[
A(n,n') := \{ f\in C([0,1], M_{n'}) : f(0) = \diag(\overbrace{c,\cdots,c}^a,0_n) \: \text{and} \: f(1) = \diag(\overbrace{c,\cdots, c}^{a+1}),  \: c\in M_n \},
\]
where $n$ and $n'$ are natural numbers with $n|n'$ and  $a :=\frac{n'}{n}-1>0$. Each building block is a  one-dimensional NCCW-complex, i.e.\ a pullback of the form
\[
\xymatrix{
A \ar[r] \ar[d] & M_n \ar[d]\\
C[0,1]\otimes M_{n'} \ar[r]_{\partial} & M_{n'} \oplus M_{n'}}
\]
(see \cite{Eilers:1998yu} and also section \ref{endomorphisms of W} below), is stably projectionless (because of the multiplicity differences at the endpoints), and has trivial $K$-theory (which can be seen from the Mayer-Vietoris sequence of \cite{Schochet:1984kx}). We will call an inductive limit of countably many building blocks a \emph{Razak algebra}. Razak algebras are separable, nuclear, satisfy the UCT and  are completely classified by tracial data when simple (Theorems \ref{class} and \ref{range}). We now establish some basic notation that will be used throughout the paper.

\subsection*{Irreducible representations.} Every ideal $I$ of $A$ is of the form $I=\{f\in A : f|_{\Gamma(I)}=0\}$ for some unique closed subset $\Gamma(I)\subseteq \mathbb{T} = [0,1]/\{0,1\}$. In particular, the primitive ideals of $A$ are precisely those of the form $I_x = \{f\in A : f(x) = 0\}$ for $x\in [0,1]$, and we can describe all the irreducible representations of $A$. They are (up to unitary equivalence) the evaluation maps $\ev_s: A \rightarrow M_{n'}$ for $s\in (0,1)$ together with `evaluation at the irreducible fibre at infinity', i.e.\ $\ev_\infty: A\rightarrow M_n$ is such that $\ev_0 = \bigoplus_{i=1}^a \ev_\infty$.

\subsection*{Connecting maps.} Given the above characterization of ideals of $A$, it is easy to see that any proper quotient of $A$ has nontrivial projections. Therefore, if $B$ is another projectionless $C^*$-algebra then every nonzero $^*$-homomorphism $\varphi: A \rightarrow B$ is injective. (This means that we may always assume that Razak algebras have injective connecting maps.) If $B$ is also a building block then for each $x\in [0,1]$ we will denote by  $\varphi^x:A\rightarrow B_x$ the morphism $\varphi^x(f) = \varphi(f)(x)$.

\subsection*{Traces.} We write $T^+A$ for the cone of densely defined lower semicontinuous traces on $A$. Every trace on the building block $A$ is bounded and has the form $\tau = \tr \otimes \mu$, where $\tr$ is the normalized trace on $M_{n'}$ and $\mu$ is some positive Borel measure on $(0,1]$. We will write $T^+_1A$ for the simplex of tracial states on $A$ (corresponding to Borel probability measures on $(0,1]$), $\Sigma_A$ for those traces of norm at most one, and we let $\aff_0T^+A$ denote the ordered vector space of continuous real-valued linear functionals on the cone $T^+A$.

Following \cite{Razak:2002kq}, we equip $\aff_0T^+A$ with two different norms. First, we define $\|\cdot\|_A$ by $\|f\|_A := \sup \{|f(\tau)| : \tau\in \Sigma_A\}$. Second, we use the order unit norm given by the following: fix $\eta \in \aff_0T^+A$ with $\inf \{\eta(\tau) : \|\tau\| = 1\} >0$, so that $\eta$ is an order unit of $\aff_0T^+A$, and denote by $\Sigma_\eta$ the closed convex set $\{\tau\in T^+A : \eta(\tau) = 1\}$; we then get a corresponding order unit norm  $\|\cdot\|_\eta$ given by $\|f\|_\eta := \sup \{|f(\tau)| : \tau \in \Sigma_\eta\}$. It is not hard to show that the norms $\|\cdot\|_A$ and $\|\cdot\|_\eta$ are equivalent for the building block $A$, so in particular $\Sigma_\eta$ is a \emph{compact} base of $T^+A$. In fact, we can say exactly what the spaces $(\aff_0 T^+A, \|\cdot\|_A)$ and $(\aff_0 T^+A, \|\cdot\|_\eta)$ look like.
\begin{enumerate}[(i)]
\item Define $C[0,1]_a := \{f\in C[0,1] : f(0) = \frac{a}{a+1} f(1)\}$. Then there is an isometric isomorphism of ordered Banach spaces $\iota: (\aff_0 T^+A, \|\cdot\|_A) \rightarrow C[0,1]_a$, $\iota(f)(t)=f(\tr\otimes\delta_t)$ (where $\delta_t$ denotes the point mass at $t$), which moreover preserves infima:
\[
\inf\{f(\tau) : \tau\in T^+_1A\} = \inf \iota(f) \quad \text{for every} \quad f\in\aff_0T^+A.
\]
\item There is an isomorphism of order unit spaces $\iota'_\eta: (\aff_0 T^+A, \eta) \rightarrow (C_\mathbb{R}(\mathbb{T}), 1)$ given by $\iota'_\eta(f) = \frac{\iota(f)}{\iota(\eta)}$, which again preserves infima: $\inf\{f(\tau) : \tau \in \Sigma_\eta\} = \inf \iota'_\eta(f)$ for $f\in \aff_0T^+A$.
\end{enumerate}

Finally, there is an embedding $\psi_A$ of $\aff_0T^+A$ into the set $A_{sa}$ of self-adjoint elements of $A$ given by
\[
\psi_A(f)(t) = \frac{a+1}{a+t}\left(\overbrace{f(\tr\otimes \delta_t)1_n \oplus \cdots \oplus f(\tr \otimes \delta_t)1_n}^a \oplus t f(\tr\otimes \delta_t)1_n\right),
\]
which is right-inverse to the usual map $\rho_A: A_{sa} \rightarrow \aff_0T^+A$. That is, $\psi_A$ satisfies $\tau(\psi_A(f)) = f(\tau)$ for every $\tau \in T^+A$. The embedding is natural in the sense that if $B$ is another building block and $\varphi: A \rightarrow B$ is a $^*$-homomorphism then
\[
\tau(\varphi \circ \psi_A(f)) = \varphi^*\tau (\psi_A(f)) = f(\varphi^* \tau) = \varphi_*(f)(\tau) = \tau(\psi_B \circ \varphi_*(f))
\]
for every $\tau\in T^+B$ and $f\in \aff_0T^+A$. This justifies suppressing the notation $\psi_A$ and $\rho_A$, and in the sequel we will do so without comment (particularly in section \ref{endomorphisms of W}).

\section{The construction of $W$} \label{construction of W}

In this section we construct explicit connecting maps for $W$ by adapting the procedure of \cite[Proposition 2.5]{Jiang:1999hb}.

\begin{proposition} \label{sequence}
There exists an inductive sequence $(A_i, \varphi_i)$ of building blocks $A_i = A(n_i, (a_i+1)n_i)$ such that  each connecting map $\varphi_{ij}: A_i \rightarrow A_j$ is a $^*$-homomorphism of the form
\begin{equation} \label{diag}
\varphi_{ij}(f) = u\left(
\begin{array}{cccc}
f\circ \xi_1 &  & & \\
& f\circ \xi_2 & & \\
&  & \ddots & \\
&  &  & f\circ \xi_m\\
\end{array}
\right) u^*
\end{equation}
for some unitary $u\in C([0,1], M_{n_j'})$ and continuous maps $\xi_k: [0,1] \rightarrow [0,1]$ that satisfy
\begin{itemize}
\item[] \begin{equation} \label{small}
|\xi_k(x) - \xi_k(y)| \le (1/2)^{j-i} \quad \text{for every}\quad x,y \in [0,1] \quad \text{and} \quad 1\le k \le m;
\end{equation}
\item[] \begin{equation}\label{small3}
\bigcup_{k=1}^m \xi_k ([0,1]) = [0,1].
\end{equation}
\end{itemize}
\end{proposition}

\begin{proof}
Let $A_1$ be some building block $A(n_1, (a+1)n_1)$. We will find a building block $A_2 = A(n_2, (b+1)n_2)$ and an injective $^*$-homomorphism $\varphi:A_1 \rightarrow A_2$ of the form (\ref{diag}) where each $\xi_k$ is one of the maps
\begin{equation} \label{list}
\xi(x) = x/2,  \qquad \xi(x) \equiv 1/2 \quad \text{or}  \quad \xi(x) = (x+1)/2.
\end{equation}
Repeating this process then gives an inductive sequence where each $\varphi_i: A_i \rightarrow A_{i+1}$ has the right form. Note also that as defined, $\varphi_i(f)$ makes sense for any $f\in C([0,1], M_{n_i'})$, so $\varphi_i$ extends to a unital $^*$-homomorphism $C([0,1], M_{n_i'}) \rightarrow C([0,1], M_{n_{i+1}'})$ and in particular $\varphi_i(u)$ is unitary whenever $u$ is. Therefore, each connecting map $\varphi_{ij}: A_i \rightarrow A_j$ will be of the form (\ref{diag}) with each $\xi_k$ a composition of $j-i$ functions from the list (\ref{list}), so will be one of the maps
\[
\xi(x) = \frac{l}{2^{j-i}} \quad \text{or} \quad \xi(x) = \frac{x+l}{2^{j-i}},
\]
for some integer $l$ with $0< l < 2^{j-i}$ in the former case and $0 \le l < 2^{j-i}$ in the latter. Hence (\ref{small}) is satisfied.

Let $b = 2a+1$, $n_2 = bn_1$ and $m=2b$. Let $f\in A_1$, so that $f(0) = \diag(\overbrace{c, \cdots, c}^a, 0_{n_1})$ and $f(1) = \diag(\overbrace{c, \cdots, c}^{a+1})$ (in $M_{(a+1)n_1}$) for some $c\in M_{n_1}$. Write $d_f:=\diag(f(1/2), \overbrace{c, \cdots, c}^a) \in M_{n_2}$. Then, in $M_{(b+1)n_2}$, the matrix
\[
d_f\otimes 1_b = \left(
\begin{array}{cccc}
d_f  &  &  & \\
 & \ddots &  & \\
 &  &  d_f &\\
 &  &  & 0_{n_2}\\
\end{array}
\right)
\]
consists (up to permutation) of $ab$ copies of $c$, $b$ copies of $f(1/2)$, and a zero matrix of size $n_2$. On the other hand, the matrix
\[
\diag(\overbrace{f(0),\ldots,f(0)}^{b}, \overbrace{f(1/2),\ldots,f(1/2)}^b) \in M_{m(a+1)n_1} = M_{(b+1)n_2}
\]
also consists of $ab$ copies of $c$, $b$ copies of $f(1/2)$, and a zero matrix of size $bn_1=n_2$. Therefore, there is a permutation unitary $u_0\in M_{(b+1)n_2}$ such that
\[
\left(
\begin{array}{cccc}
d_f  &  &  & \\
 & \ddots &  & \\
 &  &  d_f &\\
 &  &  & 0_{n_2}\\
\end{array}
\right) =
u_0\left(
\begin{array}{cccccc}
f(0) &  & & & & \\
 & \ddots &  &  &  & \\
 &  &  f(0) &  & & \\
 &  &  & f(1/2) & & \\
 &  &  &  & \ddots & \\
 &  &  &  &  & f(1/2)\\
\end{array}
\right)u_0^*.
\]
Similarly, again in $M_{(b+1)n_2}$, both of the matrices $\diag(\overbrace{f(1/2),\ldots,f(1/2)}^{b+1}, \overbrace{f(1),\ldots,f(1)}^{b-1})$ and $d_f\otimes 1_{b+1}$ consist up to permutation of $a(b+1)$ copies of $c$ and $b+1$ copies of $f(1/2)$. Therefore, there is a permutation unitary $u_1\in M_{(b+1)n_2}$ such that
\[
\left(
\begin{array}{cccc}
d_f  &  &  & \\
 & d_f &  & \\
 &  &  \ddots &\\
 &  &  & d_f \\
\end{array}
\right) =
u_1\left(
\begin{array}{cccccc}
f(1/2) &  & & & & \\
 & \ddots &  &  &  & \\
 &  &  f(1/2) &  & & \\
 &  &  & f(1) & & \\
 &  &  &  & \ddots & \\
 &  &  &  &  & f(1)\\
\end{array}
\right)u_1^*.
\]
Now we just connect the endpoints: take $u$ to be any continuous path of unitaries in $M_{(b+1)n_2}$ from $u_0$ to $u_1$, and define functions $\xi_1,\ldots,\xi_m:[0,1] \rightarrow [0,1]$ by
\[
\xi_k(x) = \left\{
\begin{array}{lll}
x/2 & \text{if} \quad 1\le k \le b,\\
1/2 & \text{if} \quad k=b+1,\\
(x+1)/2 & \text{if} \quad b+1<k\le m.
\end{array}
\right.
\]
Then the map $\varphi$ as defined in (\ref{diag}) is by construction a $^*$-homomorphism from $A_1$ to $A_2$.

Finally, note that $\varphi_{ij}$ is injective if and only if condition (\ref{small3}) holds (since $\varphi_{ij}(f) = 0$ if and only if $f\in A_i$ is supported on the open set $[0,1]\backslash \bigcup_{k=1}^m \xi_k [0,1]$). By construction, (\ref{small3}) holds for the $\xi_k$ used to define $\varphi$ above, and it therefore follows that for every $i$ and $j$, $\varphi_{ij}$ is injective and satisfies (\ref{small3}).
\end{proof}

\begin{remark} \label{growth}
For the above inductive sequence we have $a_i \rightarrow \infty$ (and also $n_i/a_i =n_{i-1} \rightarrow \infty$) as $i \rightarrow \infty$; we will make use of this in section \ref{endomorphisms of W}.
\end{remark}

\begin{lemma} \label{bounded}
If $(A_i, \varphi_i)$ is any inductive sequence as in Proposition \ref{sequence} then the connecting maps $\varphi_{ij}$ are nondegenerate. In particular, $\varphi_{ij}^*(T^+_1(A_j)) \subseteq T^+_1(A_i)$ for every $j\ge i$. Moreover, if $A = \varinjlim (A_i, \varphi_i)$ then every $\tau \in T^+A$ is bounded.
\end{lemma}

\begin{proof}
Here, `nondegenerate' means that an, and hence every, approximate unit of $A_i$ is mapped to an approximate unit of $A_j$. Let $h\in A_1$ be the canonical self-adjoint element $h(t) = (\overbrace{1\oplus \cdots \oplus 1}^a \oplus t) \otimes 1_{n_1}$. We claim that for every $i \ge 1$, $h_i : =\varphi_{1i}(h)$ is strictly positive in $A_i$, which is equivalent to saying that $(h_i^{1/n})_{n\in \mathbb{N}}$ is an approximate unit for $A_i$. To make the analysis easier, we may assume by induction that $i=2$. To prove the Lemma, let $f \in A_i$ and let $\epsilon>0$, and for convenience write $p=1_{n_i'-n_i}$ and $q=1_{n_i'}$. Certainly we have $\|h_i\|=1$, and we may assume that $\|f\|=1$ as well. Choose $\delta>0$ such that if $0\le t < \delta$ then $\|f(t)-f(0)\| < \epsilon/5$ and $\|u_t-u_0\| < \epsilon/5$ (where $u$ is as in (\ref{diag}) of Proposition \ref{sequence}). It is easy to see that as $n\rightarrow \infty$, $h_i^{1/n}(0)$ converges to $p$ and $h_i^{1/n}$ converges locally uniformly to $q$ on $(0,1]$. Hence we can find some $N$ such that $\|h_i^{1/n}(0)-p\| < \epsilon/5$ and $\|h_i^{1/n}(t)-q\| < \epsilon$ for every $\delta \le t \le 1$ and $n\ge N$. For $\delta \le t \le 1$ we therefore have
\[
\|h_i^{1/n}(t)f(t)-f(t)\| = \|(h_i^{1/n}(t)-q)f(t)\| < \epsilon \qquad \forall n\ge N.
\]
Now let  $0 \le t < \delta$ and write $g_n(t)= u_0u_t^*h_i^{1/n}(t)u_tu_0^*$. Then $g_n(t)$ commutes with $p$, and since the $\xi_k$ are increasing, we have $\|pg_n(t)p - p\| \le \|h_i^{1/n}(0) - p\| < \epsilon/5$ for $n\ge N$. Thus
\begin{align*}
\|h_i^{1/n}(t)f(t)-f(t)\| &\le \|h_i^{1/n}(t)f(t) - h_i^{1/n}(t)f(0)\| + \|h_i^{1/n}(t)f(0) - g_n(t)f(0)\|\\
&+ \|g_n(t)f(0) - pg_n(t)pf(0)\| + \|(pg_n(t)p - p)f(0)\| + \|f(0)-f(t)\|\\
&< (\epsilon + 2\epsilon + 0 + \epsilon + \epsilon)/5\\
&= \epsilon
\end{align*}
for $0 \le t < \delta$ and $n\ge N$. Therefore, $\|h_i^{1/n}f - f\| < \epsilon$ for every $n\ge N$ and hence $(h_i^{1/n})_{n\in \mathbb{N}}$ is an approximate unit for $A_i$. It follows that the connecting maps $\varphi_{ij}$ are nondegenerate.

Now suppose that $j\ge i$ and let $\rho$ be a state on $A_j$. Then $\rho\circ \varphi_{ij}$ is a positive linear functional on $A_i$ and moreover
\[
\|\rho\circ\varphi_{ij}\| = \lim_{n\rightarrow \infty} \rho\circ\varphi_{ij} (h_i^{1/n}) = \lim_{n\rightarrow \infty} \rho(\varphi_{ij}(h_i)^{1/n}) = \lim_{n\rightarrow \infty} \rho(h_j^{1/n}) = \|\rho\| = 1,
\]
so $\varphi_{ij}^*\rho = \rho\circ\varphi_{ij} \in S(A_i)$.\\

Finally, let $\tau \in T^+A$, which we identify with $\varprojlim(T^+A_i, \varphi_i^*)$. Then for every $i$, $\tau$ restricts to a bounded trace $\varphi_{i\infty}^*\tau$ on $A_i$, and we have
\[
\|\varphi_{i\infty}^*\tau\| = \lim_{n\rightarrow \infty} (\varphi_{i\infty}^*\tau) (h_i ^{1/n}) = \lim_{n\rightarrow \infty} \tau (\varphi_{i \infty}(h_i) ^{1/n}) = \lim_{n\rightarrow \infty} \tau (\varphi_{1 \infty}(h) ^{1/n}) = \|\varphi_{1\infty}^*\tau\|.
\]
Hence $\tau$ is bounded.
\end{proof}

To show that $A$ is simple, we use the following well-known lemma (see \cite[\S4]{Elliott:1997yu} for a unital version).

\begin{lemma} \label{delta}
Let $A = \varinjlim (A_i, \varphi_{ij})$ be an inductive limit of building blocks. Then $A$ is simple if and only if for every $i \in \mathbb{N}$ and every nonzero element  $a$ of $A_i$, the image $\varphi_{ij}(a)$ generates $A_j$ as a closed two-sided ideal for all but finitely many $j\ge i$.
\end{lemma}

\begin{proposition} \label{simple}
Let $A = \varinjlim (A_i, \varphi_i)$ for any inductive sequence $(A_i, \varphi_i)$ as in Proposition \ref{sequence}. Then $A$ is simple and  has a unique tracial state.
\end{proposition}

\begin{proof}
If $f$ is a nonzero element of $A_i$ then there is an interval $U\subset [0,1]$ on which $f$ is nonzero. By (\ref{small}) and (\ref{small3}) of Proposition \ref{sequence}, if $j\ge i$ is large enough then there is some $1\le k \le m$ such that $\xi_k([0,1]) \subset U$. Then $f \circ \xi_k$, and hence $\varphi_{ij}(f)$, is nonzero on all of $[0,1]$, and so $\varphi_{ij}(f)$ generates $A_j$ as a closed two-sided ideal. Lemma \ref{delta} therefore implies that $A$ is simple.

Next, note that $A$ has a nonzero trace: the cone $T^+A$ has a compact base $\Sigma$ which can be written as the inverse limit of bases $\Sigma_i$ of $T^+A_i$; since these are nonempty compact Hausdorff spaces it follows that $\Sigma$ is nonempty (and does not contain zero since it is a base). By Lemma \ref{bounded}, this trace is bounded and we now show that it is unique. Let $f\in A_i$ and $\epsilon>0$ be given and choose $\delta>0$ such that $\|f(y) - f(z)\| \le \epsilon/2$ whenever $|y-z|\le \delta$. Then provided $2^{j-i} > 1/\delta$, it follows from (\ref{small}) that $\|f(\xi_k(x)) - f(\xi_k(y))\| \le \epsilon/2$ for every $x,y \in [0,1]$ and $1\le k \le m$. It follows that for all sufficiently large $j$, every $\tau = \tr \otimes \mu \in T^+_1A_j$ and for fixed $y\in (0,1]$, we have
\begin{equation} \label{6}
|\tau(\varphi_{ij}(f)) - \tr \otimes \delta_y(\varphi_{ij}(f))|  = \left |\int \tr(\varphi_{ij}(f)(x) - \varphi_{ij}(f)(y)) d\mu(x) \right |    \le \epsilon/2
\end{equation}
and so
\[
|\tau_{j,1}(\varphi_{ij}(f)) - \tau_{j,2}(\varphi_{ij}(f))| \le \epsilon \quad \text{for every} \quad \tau_{j,1}, \tau_{j,2} \in T^+_1A_j.
\]
Hence $A$ has at most one tracial state.
\end{proof}

As in the introduction, we will denote the unique such inductive limit by $W$ and its unique tracial state by $\tau$.

\section{A categorical description of $W$} \label{endomorphisms of W}

In this section we characterize $(W, \tau)$ as a terminal object (Theorem \ref{terminal}), and use this description to prove that every trace-preserving endomorphism of $W$ is approximately inner (Corollary \ref{inner}), and that every simple Razak algebra embeds into $W \otimes \mathcal{K}$ (Corollary \ref{embedding}). We accomplish this via the following adaptation of \cite[Theorem 2.1]{Rordam:2004kq}.

\begin{lemma} \label{tracial embedding}
Let $B$ be a building block and let $\tau$ be the unique tracial state on $W$.
\begin{enumerate}[(i)]
\item For every faithful trace $\tau_0$ on $B$ with $\|\tau_0\|\le1$ there exists a $^*$-homomorphism $\psi : B \rightarrow W$ such that $\tau \circ \psi = \tau_0$.
\item Two $^*$-homomorphisms $\psi_1, \psi_2 : B \rightarrow W$ are approximately unitarily equivalent if and only if $\tau\circ\psi_1 = \tau\circ\psi_2$.
\end{enumerate}
\end{lemma}

To prove this, we need to use Razak's local existence and local uniqueness theorems, which appear as \cite[Theorem 3.1]{Razak:2002kq} and \cite[Theorem 4.1]{Razak:2002kq} respectively. We restate them here for convenience, referring to section \ref{building blocks} for notation.

\begin{proposition}[Local existence] \label{local existence}
Let $B$ be a building block, and fix some finite subset $F \subset \aff_0T^+B$ and some $\epsilon>0$. Then there is a natural number $N$ and some $\eta \in \aff_0T^+B$ with $\|\eta\|_B \le 1$ and $\inf \iota(\eta) \ge 1/2$ such that the following property holds. For any building block $A=A(n,(a+1)n)$ and contractive positive linear map $\xi: (\aff_0T^+B, \|\cdot\|_B) \rightarrow (\aff_0T^+A, \|\cdot\|_A)$, if $n \ge Na/\inf \iota(\xi(\eta))$ then there is a $^*$-homomorphism $\psi: B \rightarrow A$ with $\|\xi(f) - \psi_*(f)\|_{\xi(\eta)} < \epsilon$ for every $f\in F$.
\end{proposition}

\begin{proposition}[Local uniqueness] \label{local uniqueness}
Let $B$ be a building block and let $h$ be the canonical self-adjoint element of $B$ (as in Lemma \ref{bounded}). Fix a finite subset $F \subset B$ and a tolerance $\epsilon >0$. Then there exists a natural number $M$ and two families of positive functions $\{\zeta_j\}_{j=1}^M, \{\sigma_j\}_{j=1}^M$ in the unit ball of $(\aff_0 T^+B,\|\cdot\|_B)$ such that for any building block $A$ and any two $^*$-homomorphisms $\varphi, \psi : B \rightarrow A$ that satisfy
\begin{enumerate}[(i)]
\item $\varphi_*(\zeta_j)(\tau) >m$, $\psi_*(\zeta_j)(\tau) >m$, and $|\varphi_*(\sigma_j)(\tau) - \psi_*(\sigma_j)(\tau)| < m$ for some $m>0$ and every $\tau \in T^+_1A$ and $1 \le j \le M$;
\item $\varphi^t(h)$ and $\psi^t(h)$ have at least three distinct eigenvalues for every $t\in [0,1]$;
\end{enumerate}
there exists a unitary $u \in \widetilde A$ such that $\|\varphi(f) - u \psi(f) u^* \| < \epsilon$ for every $f \in F$.
\end{proposition}
 
Typically, the eigenvalue condition (ii) of Proposition \ref{local uniqueness} is handled by the following standard consequence of Lemma \ref{delta}.

\begin{lemma}[$\delta$-density] \label{eigenvalues}
Let $A = \varinjlim (A_i, \varphi_{ij})$ be a simple inductive limit of building blocks $A_i = A(n_i, (a_i+1)n_i)$, and let $h\in A_1$ be the canonical self-adjoint element $h(t) = (\overbrace{1\oplus \cdots \oplus 1}^a \oplus t) \otimes 1_{n_1}$. Then for every $\delta>0$, there exists an integer $N$ such that for every $j\ge N$ and every $x\in [0,1]$, the eigenvalues of $\varphi_{1j}^x(h)$ are $\delta$-dense in $[0,1]$.
\end{lemma}

Note that Lemma \ref{eigenvalues} implies that $n_i \rightarrow \infty$ as $i \rightarrow \infty$ for any simple inductive limit of building blocks $A_i = A(n_i, (a_i+1)n_i)$. Of course, we already know this for $W$ by construction; either way, this will allow us to deal with the hypothesis of Proposition \ref{local existence} (see also Remark \ref{growth}).

We also need to use the fact that, since the building blocks are (nonunital) one-dimensional NCCW complexes, they are semiprojective, and hence can be finitely presented with stable relations (see \cite{Eilers:1998yu}). We can therefore use \cite[Lemma 3.7]{Loring:1993kq} to deduce that whenever $B$ is a building block, $\theta: B \rightarrow C=\overline{\bigcup_{i=1}^\infty C_i}$ is a $^*$-homomorphism and we have fixed some finite subset $F\subset B$ and some tolerance $\epsilon>0$,  for all sufficiently large $k\in\mathbb{N}$ there exists a $^*$-homomorphism $\psi: B \rightarrow C_k$ such that $\|\theta(f)-\psi(f)\| < \epsilon$ for every $f\in F$. We will use this in the proof of Lemma \ref{tracial embedding} (ii).
 
\begin{proof}[Proof of Lemma \ref{tracial embedding}]
As usual, write $W=\overline{\bigcup_{i=1}^\infty A_i}$ with $A_i=A(n_i,(a_i+1)n_i)$.\\
(i) Let $\tau_0$ be a faithful trace in $\Sigma_B$ and fix an increasing sequence $F_1 \subset F_2 \subset \cdots$ of finite sets of self-adjoint elements of the unit ball $B_1$ of $B$ such that $\bigcup_{k=1}^\infty F_k$ is dense in the self-adjoint part of $B_1$. We will find a sequence $(i_k)_{k=1}^\infty$, together with $^*$-homomorphisms $\psi_k: B \rightarrow A_{i_k}$ and unitaries $u_k \in \widetilde{A_{i_k}}$ (with $u_1=1$) such that
\[
\|\psi_k(f) - u_{k+1}\psi_{k+1}(f)u_{k+1}^*\| < 2^{-k} \quad \text{and} \quad |\tau(\psi_k(f)) - \tau_0(f)| < 1/k
\]
for every $f\in F_k$. We will then have an approximately commutative diagram
\[
\xymatrix @C= 4pc @R= 4pc{
B \ar[r]^{\id} \ar[d]^{\psi_1} & B \ar[r]^{\id} \ar[d]^{\ad_{u_2} \circ \psi_2} & B \ar[r]^{\id} \ar[d]^{\ad_{u_2u_3}\circ \psi_3} & \cdots \ar[r] & B \ar@{-->}[d]^\psi\\
A_{i_1} \ar[r]  & A_{i_2} \ar[r]  & A_{i_3} \ar[r] & \cdots \ar[r] & W}
\]
such that $|\tau(\ad_{u_1\cdots u_k} \circ \psi_k(f)) - \tau_0(f)| < 1/k$ for every $f\in F_k$, and we get a $^*$-homomorphism $\psi:B \rightarrow W$ that satisfies $\psi(f) = \lim_{k\rightarrow \infty} u_1\cdots u_k\psi_k(f)u_k^*\cdots u_1^*$ for every $f\in B$. This will imply that $\tau\circ\psi = \tau_0$.

The idea is to use local existence (Proposition \ref{local existence}) to find the $\psi_k$ and local uniqueness (Proposition \ref{local uniqueness}) to find the $u_k$. Working inductively, fix $k\ge1$, and let $\{\zeta_j\}_{j=1}^M, \{\sigma_j\}_{j=1}^M \subset \aff_0 T^+B$ be the test functions in Proposition \ref{local uniqueness} corresponding to the finite set $F_k$ and the tolerance $\epsilon_k=2^{-k}$. Define $G_k$ to be the finite set $F_k \cup \{\zeta_j\}_{j=1}^M \cup \{\sigma_j\}_{j=1}^M \cup G_{k-1} \subset B_{sa}$ (with $G_0:=\emptyset$). Set $c_k := 2/3\min\{\tau_0(\zeta_j): 1\le j \le M\}$ and $\delta_k := \min\{ 1/k, c_k/2, c_{k-1}/2\}$. Since $\tau_0$ is faithful, we have $c_k>0$ and $\delta_k>0$.

Let $\eta_k \in \aff_0T^+B$ and $N_k\in\mathbb{N}$ be as in Proposition \ref{local existence}, corresponding to the finite set $G_k$ and the tolerance $\delta_k/2$. For each $i\in \mathbb{N}$, fix $\nu_i \in \aff_0T^+A_i \cong C[0,1]_{a_i}$ with, say, $\|\nu_i\|_{A_i}=1$ and $\inf \iota(\nu_i) = a_i/(a_i+1)$. By construction of $W$ (Proposition \ref{sequence}, see also Remark \ref{growth}) we have $a_i,n_i/a_i \rightarrow \infty$ as $i \rightarrow \infty$, so we may choose $i_k>i_{k-1}$ (where $i_0:=1$) such that $n_{i_k}/a_{i_k}>4N_k/\|\tau_0\|$ and $a_{i_k}/(a_{i_k}+1)>1-\delta_k/2$. 

Define $\xi_k: \aff_0T^+B \rightarrow \aff_0T^+A_{i_k}$ by $\xi_k(f)(\tau'):= \nu_{i_k}(\tau')f(\tau_0)$. This $\xi_k$ is positive and linear, and we have
\[
\|\xi_k(f)\|_{A_{i_k}} = \sup\{|\xi_k(f)(\tau')| : \|\tau'\|=1\} = \|\nu_{i_k}\|_{A_{i_k}}|f(\tau_0)| \le \|f\|_B
\]
(since $\|\tau_0\|\le1$), so $\xi_k$ is a contraction from $(\aff_0T^+B, \|\cdot\|_B)$ to $(\aff_0T^+A_{i_k}, \|\cdot\|_{A_{i_k}})$. Since
\[
\inf\iota(\xi_k(\eta_k)) = \inf_{\tau' \in T^+_1A_{i_k}}\xi_k(\eta_k)(\tau') = \eta_k(\tau_0) \inf \iota(\nu_{i_k}) \ge \frac{\eta_k(\tau_0)}{2} \ge \frac{\|\tau_0\|}{4}\ge \frac{N_ka_{i_k}}{n_{i_k}},
\]
Proposition \ref{local existence} implies that there exists a $^*$-homomorphism $\psi_k: B \rightarrow A_{i_k}$ such that $\|\xi_k(f)-(\psi_k)_*(f)\|_{\xi_k(\eta_k)} < \delta_k/2$ for every $f\in G_k$. Moreover, we have
\[
|\nu_{i_k}(\tau')-1| \le 1-a_{i_k}/(a_{i_k}+1) < \delta_k/2
\]
for every $\tau'\in T^+_1A_{i_k}$. Hence
\begin{align*}
\delta_k/2 &> \sup_{\xi_k(\eta_k)(\tau')=1}|\nu_{i_k}(\tau')f(\tau_0) - (\psi_k)_*(f)(\tau')|\\
&=  \sup_{\nu_{i_k}(\tau')=1/\eta_k(\tau_0)} |\nu_{i_k}(\tau')f(\tau_0) - (\psi_k)_*(f)(\tau')|\\
&\ge \sup_{\|\tau'\|=1} |\nu_{i_k}(\tau')f(\tau_0) - (\psi_k)_*(f)(\tau')|\\
&\ge \sup_{\|\tau'\|=1} |f(\tau_0) - (\psi_k)_*(f)(\tau')| - \delta_k/2
\end{align*}
for every $f\in G_k$. (For the penultimate inequality we have used the fact that $\eta_k(\tau_0)\le1$ and $\|\nu_{i_k}\|_{A_{i_k}}=1$.) Thus
\begin{equation}
|\tau'(\psi_k(f)) - \tau_0(f)| < \delta_k \quad \forall f\in G_k \quad \forall \tau'\in T^+_1A_{i_k}.
\end{equation}
In particular, noting Lemma \ref{bounded} and its proof, $|\tau(\psi_k(f))-\tau_0(f)| < 1/k$ for every $f\in F_k$, and for every $\tau'\in T^+_1A_{i_{k+1}} \subseteq T^+_1A_{i_k}$ and $1\le j \le M$ we have
\[
(\psi_k)_*(\zeta_j)(\tau') > c_k, \quad (\psi_{k+1})_*(\zeta_j)(\tau')  > c_k, \quad \text{and} \quad |(\psi_{k+1})_*(\sigma_j)(\tau') - (\psi_k)_*(\sigma_j)(\tau')| < c_k.
\]
By Lemma \ref{eigenvalues}, we may also assume (by replacing each $\psi_k$ by $\varphi_{i_k,l}\circ\psi_k$ as necessary) that for every $k$, $\psi_k^t(h)$ has at least three distinct eigenvalues for every  $t\in[0,1]$. Hence, by Proposition \ref{local uniqueness}, there exists a unitary $u_{k+1} \in \widetilde{A_{i_{k+1}}}$ such that $\|\psi_k(f) - u_{k+1} \psi_{k+1}(f) u_{k+1}^* \| < 2^{-k}$ for every $f \in F_k$, as required.

(ii) The `only if' part is obvious. Suppose conversely that $\psi_1, \psi_2: B \rightarrow W$ are $^*$-homomorphisms with $\tau\circ\psi_1=\tau\circ\psi_2$. If either map is zero, then the statement is trivial, so we may assume that both $\psi_1$ and $\psi_2$ are injective (see section \ref{building blocks}). Fix a finite subset $F\subset B_{sa}$ and a tolerance $\epsilon>0$, and let $\{\zeta_j\}_{j=1}^M, \{\sigma_j\}_{j=1}^M \subset \aff_0 T^+B$ be the test functions in Proposition \ref{local uniqueness} corresponding to $F$ and $\epsilon/3$. Set $F':= F \cup \{\zeta_j\}_{j=1}^M \cup \{\sigma_j\}_{j=1}^M$ and $\delta:= \min\{\epsilon/3, \tau(\psi_1(\zeta_j))/6 : 1\le j \le M\}$; since $\psi_1$ and $\psi_2$ are injective, we have $\delta>0$. By semiprojectivity of $B$, for all sufficiently large $k$ there exist $^*$-homomorphisms $\psi_1^{(k)}, \psi_2^{(k)}: B \rightarrow A_k$ such that
\begin{equation} \label{7.5}
\|\psi_m^{(k)}(f) - \psi_m(f)\| < \delta \qquad \forall f\in F', \: m=1,2.
\end{equation}
Note in particular that
\begin{equation} \label{8}
|\tau\circ\psi_1^{(k)}(f) - \tau\circ\psi_2^{(k)}(f)| < 2\delta \qquad \forall f\in F'.
\end{equation}
We may also assume that $k$ is large enough so that
\begin{equation} \label{8.5}
\sup\{|\tau(x) - \tau'(x)| : \tau'\in T^+_1A_k\} < \delta \qquad \forall x\in \psi_1^{(k)}(F') \cup \psi_2^{(k)}(F')
\end{equation}
(as in (\ref{6}) of Proposition \ref{simple}) and so that condition (ii) of Proposition \ref{local uniqueness} holds. By (\ref{8}) and (\ref{8.5}) we have $|(\psi_1^{(k)})_*(\sigma_j)(\tau') - (\psi_2^{(k)})_*(\sigma_j)(\tau')| < 4\delta$ and by (\ref{8.5}), (\ref{7.5}) and our choice of $\delta$ we have $(\psi_m^{(k)})_*(\zeta_j)(\tau') > 4\delta
$ for $m=1,2$, $1\le j\le M$ and $\tau' \in T^+_1A_k$. Hence, by Proposition \ref{local uniqueness}, there exists a unitary $u \in \widetilde{A_k}$ such that $\|\psi_1^{(k)}(f) - u \psi_2^{(k)}(f) u^* \| < \epsilon/3$ for every $f \in F$, which implies that
\[
\|\psi_1(f) - u\psi_2(f)u^*\| < \epsilon \qquad \forall f\in F.
\]
This proves that $\psi_1$ and $\psi_2$ are approximately unitarily equivalent.
\end{proof}

An object $T$ in a category $\mathcal{C}$ is \emph{terminal} if for every object $X$ in $\mathcal{C}$ there exists a unique morphism from $X$ to $T$; such objects are unique up to isomorphism. For example, the Cuntz algebra $\mathcal{O}_2$ is the unique terminal object in the category of strongly self-absorbing $C^*$-algebras, where the morphisms are approximate unitary equivalence classes of unital $^*$-homomorphisms (see \cite{Kirchberg:2000kq} and \cite{Toms:2007uq}). We now use Lemma \ref{tracial embedding} and an intertwining argument to characterize $(W,\tau)$ as a terminal object.

\begin{theorem} \label{terminal}
$(W,\tau)$ is the unique terminal object in the category whose objects are pairs $(A,\tau_A)$ with $A$ a simple Razak algebra and $\tau_A\in \Sigma_A$, and where a morphism from $(A,\tau_A)$ to $(B,\tau_B)$ is (the approximate unitary equivalence class of) a $^*$-homomorphism $\psi:A\rightarrow B$ with $\psi^*\tau_B=\tau_A$. 
\end{theorem}

\begin{proof}
Let $B = \varinjlim (B_i,\beta_i)$ be a simple Razak algebra and let $\tau_0 \in \Sigma_B$. We need to show that there is a $^*$-homomorphism $\psi:B\rightarrow W$ with $\psi^*\tau=\tau_0$, and prove that, up to approximate unitary equivalence, $\psi$ is the unique map $B\rightarrow W$ with this property. This is obvious if $\tau_0=0$ (since $\tau$ is faithful), so we may assume that $\tau_0$ is nonzero, hence faithful (since $B$ is simple). Write $\tau_0=(\tau_i)_{i=1}^\infty$, where for $i\in\mathbb{N}$, $\tau_i$ is a faithful trace on $B_i$ with $\|\tau_i\|\le1$ and $\tau_{i+1}\circ\beta_i = \tau_i$. By Lemma \ref{tracial embedding}(i), for each $i$ there exists a $^*$-homomorphism $\psi_i: B_i \rightarrow W$ with $\tau \circ \psi_i = \tau_i$. But $\tau\circ\psi_{i+1}\circ\beta_i = \tau_{i+1}\circ\beta_i = \tau_i$, so by Lemma \ref{tracial embedding}(ii), we have $\psi_{i+1} \circ \beta_i \sim_{a.u.} \psi_i$. It then follows from a (one-sided) approximate intertwining (as in the proof of Lemma \ref{tracial embedding}(i)) that there is a $^*$-homomorphism $\psi: B \rightarrow W$  which, by construction, satisfies $\psi^*\tau = \tau_0$. By restricting $\psi$ to the building blocks $B_i$, Lemma \ref{tracial embedding}(ii) says that, up to approximate unitary equivalence, $\psi$ is the unique $^*$-homomorphism $B\rightarrow W$ with this property.
\end{proof}

The next two results are immediate corollaries of Theorem \ref{terminal}.

\begin{corollary} \label{inner}
Every trace-preserving endomorphism (hence every nondegenerate endomorphism) of $W$ is approximately inner. That is, for any such endomorphism $\theta$, there is a sequence of unitaries $u_n$ in $\widetilde W$ such that $\theta(a) = \lim_{n \rightarrow \infty} u_n a u_n^*$ for every $a \in W$.
\end{corollary}

\begin{corollary} \label{embedding}
Let $B$ be a simple Razak algebra. Then $B$ admits a tracial state if and only if $B$ is isomorphic to a subalgebra of $W$. If $B$ has no nonzero bounded traces, then $B$ is stable and is isomorphic to a subalgebra of $W \otimes \mathcal{K}$.
\end{corollary}

\begin{proof}
The first assertion follows from Theorem \ref{terminal}. For the second assertion, note that if $B$ is a simple Razak algebra then so is $B\otimes \mathcal{K}$, and $(T^+(B \otimes \mathcal{K}), \Sigma_{B \otimes \mathcal{K}}) \cong (T^+B, 0)$. Therefore, Theorem \ref{class} implies that $B \cong B \otimes \mathcal{K}$ whenever $B$ has no nonzero bounded traces, and Theorems \ref{range} and \ref{class} imply that \emph{every} simple Razak algebra is stably isomorphic to a simple Razak algebra which has no unbounded traces. The second statement therefore follows from the first.
\end{proof}

\section{$W$-stability} \label{W is self-absorbing}

Let us first make a remark on the proof that $\mathcal{Z}$ is strongly self-absorbing. Jiang and Su adopt the following strategy.
\begin{enumerate}[(i)]
\item Prove that the two maps $\id \otimes 1, 1\otimes \id: \mathcal{Z} \rightarrow \mathcal{Z} \otimes \mathcal{Z}$ are approximately unitarily equivalent.
\item Show that there exists a unital $^*$-homomorphism $\psi: \mathcal{Z} \otimes \mathcal{Z} \rightarrow \mathcal{Z}$.
\item Prove that $(\id \otimes 1) \circ \psi \sim_{a.u.} \id_{\mathcal{Z} \otimes \mathcal{Z}}$ and note that $\psi \circ (\id \otimes 1) \sim_{a.u.} \id_{\mathcal{Z}}$ since every unital endomorphism of $\mathcal{Z}$ is approximately inner. A standard intertwining argument \cite[Proposition A]{Rordam:1994rm} then shows that there exists an isomorphism $\varphi: \mathcal{Z} \rightarrow \mathcal{Z} \otimes \mathcal{Z}$. Again, since every unital endomorphism of $\mathcal{Z}$ is approximately inner, it follows easily that $\varphi \sim_{a.u.} \id \otimes 1$.
\end{enumerate}

Step (ii) in this procedure goes roughly as follows: Write the Jiang-Su algebra as $\mathcal{Z} = \varinjlim (A_n, \varphi_n)$, where each $A_n = I[p_n, d_n, q_n]$ is a prime dimension drop algebra (i.e.\ $p_n$ and $q_n$ are coprime, $d_n = p_nq_n$ and $A_n = \{f \in C([0,1], M_{d_n}) : f(0) \in M_{p_n} \otimes 1_{q_n}, f(1) \in 1_{p_n} \otimes M_{q_n}\}$). Define $B_n$ to be the diagonal of $A_n \otimes A_n$, i.e.\ $B_n$ consists of all continuous functions $f : [0,1] \rightarrow M_{d_n^2}$ such that $f(0) \in (M_{p_n} \otimes 1_{q_n})^{\otimes 2}$ and $f(1) \in (1_{p_n} \otimes M_{q_n})^{\otimes 2}$. Then $B_n \cong I[p_n^2, d_n^2, q_n^2]$ is a prime dimension drop algebra and we have a $^*$-homomorphism $\rho_n : A_n \otimes A_n \rightarrow B_n$ given by restriction: $\rho_n(f)(x) = f(x,x)$ for $f \in A_n \otimes A_n$ and $x\in [0,1]$. Jiang and Su construct connecting maps $\psi_n : B_n \rightarrow B_{n+1}$ such that the diagram
\[
\xymatrix @C= 4pc @R= 4pc{
A_1 \otimes A_1 \ar[r]^{\varphi_1 \otimes \varphi_1} \ar[d]^{\rho_1} & A_2 \otimes A_2 \ar[r]^{\varphi_2 \otimes \varphi_2} \ar[d]^{\rho_2} & A_3 \otimes A_3 \ar[r] \ar[d]^{\rho_3} & \cdots \ar[r] & \mathcal{Z} \otimes \mathcal{Z} \ar@{-->}[d]^\rho\\
B_1 \ar[r]^{\psi_1}  & B_2 \ar[r]^{\psi_2}  & B_3 \ar[r] & \cdots \ar[r] & B}
\]
commutes approximately (where $B := \varinjlim (B_n, \psi_n)$), so there is an induced morphism $\rho: \mathcal{Z} \otimes \mathcal{Z} \rightarrow B$. They also show that $B$ is simple, and since it is an inductive limit of (prime) dimension drop algebras, their Theorem 6.2 then shows that there exists a morphism from $B$ to $\mathcal{Z}$; the composition of this morphism with $\rho$ is the required morphism $\psi : \mathcal{Z} \otimes \mathcal{Z} \rightarrow \mathcal{Z}$.

It turns out that, even though all of the restriction maps $\rho_j$ are surjective, the induced morphism $\rho$ need not be. In particular, $B$ may not be isomorphic to $\mathcal{Z} \otimes \mathcal{Z}$, and steps (i) and (iii) seem to be unavoidable. This is unfortunate for us because it is difficult to make sense of these steps for the nonunital algebra $W$. Perhaps it would be possible to construct a \emph{two-sided} approximate intertwining between $W \otimes W = \varinjlim (A_i \otimes A_i, \varphi_i \otimes \varphi_i)$ and some limit of suitable one-dimensional NCCW complexes, but it is currently unclear how this might be done. We therefore leave the following as a conjecture, and remark that classification certainly predicts it to be true.

\begin{conjecture} \label{absorbing}
$W \otimes W \cong W$.
\end{conjecture}

\begin{remark} \label{flip}
If Conjecture \ref{absorbing} were to hold, then it would follow immediately from Corollary \ref{inner} that $W$ would have approximately inner flip, i.e.\ there would be a sequence $(u_n)_{n=1}^\infty$ of unitaries in the unitization of $W\otimes W$ such that $\lim_{n\rightarrow \infty} u_n(a \otimes b)u_n^* = b \otimes a$ for  $a, b \in W$.
\end{remark}

Next, we prove that, were Conjecture \ref{absorbing} to hold, every simple Razak algebra would absorb $W$ tensorially. We need the following extension of a theorem of Blackadar \cite{Blackadar:1980zr} and Goodearl \cite{Goodearl:1977dq}, which is likely to be already known by experts. The reader is referred to \cite{Brown:2008mz} or \cite{Ara:2009cs} for details of the Cuntz semigroup $\cu(\cdot)$.

\begin{proposition} \label{af main}
Let $\Delta$ be a metrizable Choquet simplex, let $C$ be the cone with $\Delta$ as its base, and let $\omega$ be a lower semicontinuous affine map $\Delta \rightarrow (0,\infty]$. Then there exists a simple AF algebra $A$ and an isomorphism $T^+A \cong C$ under which $\omega$ corresponds to the norm map on $T^+A$.
\end{proposition}

\begin{proof}
By \cite[Theorem 3.10]{Blackadar:1980zr}, there exists a simple, unital AF algebra $B$ with $T^+_1B=\Delta$. (Tensoring $B$ with a UHF algebra if necessary, we may assume that $B$ is infinite dimensional.) We will produce a hereditary sublagebra $A$ of $B\otimes\mathcal{K}$ with the required properties. Since every trace on $B$ extends uniquely to one on $B\otimes\mathcal{K}$, we identify $T^+(B\otimes\mathcal{K})$ with $C$. As in \cite{Brown:2007rz}, denote by $\saff(\Delta)$ the space of strictly positive lower semicontinuous affine maps on $\Delta$. The key fact is that the natural map $\cu(B)\backslash V(B) \rightarrow \saff(T^+_1B)$ is surjective. What this means is that, given $\omega\in \saff(\Delta)$, there exists a positive element $b\in B\otimes\mathcal{K}$ (which we may assume to be of norm $1$) such that $\omega(\tau)=\lim_{n\rightarrow\infty} \tau(b^{1/n})$ for every $\tau\in \Delta$. This follows from \cite[Theorem 5.3]{Brown:2008mz} and \cite[Theorem 2.6]{Brown:2007rz}, but is not difficult to prove using the well-documented structure of the $K$-theory of simple, unital AF algebras.

Let $A=\overline{b(B\otimes\mathcal{K})b}$. Then $A$ is a (full) hereditary subalgebra of $B\otimes\mathcal{K}$, so is a simple AF algebra which is stably isomorphic to $B$ (for example by \cite[Theorem 2.8]{Brown:1977kq} and \cite[Theorem 3.1]{Elliott:1976zr}). Moreover, every trace on $A$ extends uniquely to a trace on $B\otimes\mathcal{K}$ (see \cite[Remark 2.27 (viii)]{Blanchard:2004kx}), so $T^+A$ can also be identified with $C$. Since $(b^{1/n})_{n=1}^\infty$ is an approximate unit for $A$, we then have
\[
\|\tau\| = \lim_{n\rightarrow\infty} \tau(b^{1/n}) = \omega(\tau)
\]
for every $\tau\in\Delta$ (and hence, by linearity, for every $\tau\in T^+A$).
\end{proof}

\begin{corollary} \label{w-stability}
Let $A$ be a simple inductive limit of one-dimensional NCCW complexes $A_i$ such that $K_1(A_i)=0$ for every $i$ and $K_0(A)=0$. Then there exists a simple AF algebra $B$ such that $A \cong B \otimes W$. In particular, if Conjecture \ref{absorbing} were to hold, it would follow that $A \otimes W \cong A$.
\end{corollary}

\begin{proof}
Taking $C=T^+A$ and $\omega=\|\cdot\|$, Theorem \ref{af main} gives a simple AF algebra $B$ and an isomorphism between $T^+B$ and $T^+A$ that maps the tracial states of $B$ onto the tracial states of $A$. The $C^*$-algebra $W$ is by construction a simple inductive limit of one-dimensional NCCW complexes each of which has trivial $K$-theory, so $B\otimes W$ is also of this form. Moreover, $T^+W$ is generated by a tracial state, so there is an isomorphism between $T^+(B\otimes W)$ and $T^+B$ that maps the tracial states of $B\otimes W$ onto the tracial states of $B$. (This is easy to see for $F\otimes W$ whenever $F$ is finite dimensional, and the assertion for $B$ follows from the continuity of the functor $T^+$.) Theorem \ref{class2} then implies that $A \cong B \otimes W$, and the second statement would follow directly from Conjecture \ref{absorbing}.
\end{proof}

\begin{remark}
This shows that every inductive limit as in the statement of Corollary \ref{w-stability} is isomorphic to an inductive limit of finite direct sums of Razak building blocks. Moreover, the full range of the invariant is exhausted, so we obtain a version of Theorem \ref{range}.
\end{remark}

In the opposite direction, we note that a certain dichotomy holds for well-behaved, simple, $W$-stable $C^*$-algebras.

\begin{proposition} \label{wstable}
Suppose that $A$ is a simple, separable, nuclear $C^*$-algebra that satisfies the UCT, such that $A\otimes W \cong A$. Then either $A \cong \mathcal{O}_2 \otimes \mathcal{K}$ or $A$ is stably projectionless.
\end{proposition}

\begin{proof}
Since $W$ lies in the UCT class, we can apply the K\"unneth Theorem to deduce that $K_*(A) = K_*(A \otimes W)=0$. Suppose that $A \otimes \mathcal{K}$ contains a nonzero projection $p$, and let $B := p(A\otimes \mathcal{K})p$. Then $B$ is a unital $C^*$-algebra which, by  Brown's Theorem \cite[Theorem 2.8]{Brown:1977kq},  is stably isomorphic to $A$; in particular, $K_*(B)=0$. Hence $[p]=[0]$ in $K_0(B)$, so $p\oplus q \sim 0 \oplus q$ for some projection $q \in M_n(B)$. Hence $p \oplus q$ is infinite, so $B$ is `stably infinite' and so is $A$. Then, since $A \cong A \otimes W$ is tensorial non-prime, $A$ must in fact be purely infinite \cite[Theorem 4.1.10]{Rordam:2002yu} and stable \cite[Theorem 4.1.3]{Rordam:2002yu}. The Kirchberg-Phillips classification theorem \cite[Theorem 8.4.1]{Rordam:2002yu} then shows that $A \cong \mathcal{O}_2 \otimes \mathcal{K}$.
\end{proof}

\section{Asymptotically central sequences} \label{W is approx divisible}

If $A$ and $B$ are $C^*$-algebras and $C$ is a sub-$C^*$-algebra of $B$ then a sequence of $^*$-homomorphisms $\varphi_n: A \rightarrow B$ is said to be \emph{asymptotically central for $C$} if $\|[\varphi_n(a), c]\| \rightarrow 0$ as $n \rightarrow \infty$ for every $a\in A$ and $c\in C$. Such a sequence induces a $^*$-homomorphism $\varphi: A \rightarrow B_\infty \cap C'$. Here, $B_\infty$ is the limit algebra $l_\infty(B)/c_0(B)$, and there is a canonical inclusion of $B$ into $B_\infty$ as constant sequences; $B_\infty \cap C'$ denotes the relative commutant of $C$ in $B_\infty$.

Every strongly self-absorbing $C^*$-algebra $\mathcal{D}$ admits an asymptotically central sequence of unital endomorphisms \cite[Proposition 1.10]{Toms:2007uq}. Conversely, exhibiting asymptotically central $^*$-homomorphisms can be used to prove $\mathcal{D}$-stability---see \cite{Rordam:1994rm}, \cite[Theorem 7.2.2]{Rordam:2002yu} and \cite[Theorem 2.3]{Toms:2007uq}. It is not clear whether there exists such a relationship between asymptotically central sequences and tensor products in the nonunital case, essentially because we lack the notion of an infinite tensor product. Nevertheless, asymptotically central sequences are still interesting and useful in their own right, and we show below that we can adapt the approximate divisibility of $W$ to at least find an embedding of $W$ into the limit algebra $M(W)_\infty \cap W'$.

\begin{definition} \label{div}
A $C^*$-algebra $A$ is said to be \emph{approximately divisible} if for any $N\in \mathbb{N}$ there is a sequence of unital $^*$-homomorphisms $\mu_n$ from $M_N \oplus M_{N+1}$ into the multiplier algebra $M(A)$ which is asymptotically central for $A$.
\end{definition}

Approximate divisibility for unital $C^*$-algebras was studied in \cite{Blackadar:1992qy}, and the nonunital case appears for example in \cite[Definition 3.1.10]{Rordam:2002yu}.  Toms and Winter prove in \cite[\S2]{Toms:2005kq} that separable approximately divisible $C^*$-algebras are $\mathcal{Z}$-stable. Moreover, their Proposition 4.1 says that every simple Razak algebra is approximately divisible (hence $\mathcal{Z}$-stable). We show below that this is in fact an immediate consequence of classification (Theorem \ref{class}).

Let $\mathcal{Q}$ denote the universal UHF algebra (characterized by $K_0(\mathcal{Q}) = \mathbb{Q}$); $\mathcal{Q}$ is isomorphic to its infinite tensor product $\mathcal{Q}^{\otimes \infty}$ and (so) is strongly self-absorbing.

\begin{proposition} \label{uhf}
Every simple Razak algebra is $\mathcal{Q}$-stable (so absorbs every UHF algebra).
\end{proposition}

\begin{proof}
Let $B = \varinjlim (B_i, \beta_i)$ be a simple Razak algebra and let $U = \varinjlim (M_{k_i}, \alpha_i)$ be a UHF algebra. Note that if $A$ is a building block then so is $A \otimes M_k$ for every $k$, so $B \otimes U \cong \varinjlim (B_i \otimes M_{k_i}, \beta_i \otimes \alpha_i)$ is also a simple Razak algebra. Moreover, $(T^+(B \otimes U), \Sigma_{B \otimes U}) \cong (T^+B, \Sigma_B)$ since $U$ has a unique tracial state. Therefore, Theorem \ref{class} implies that $B \otimes U \cong B$.
\end{proof}

\begin{corollary} \label{divcor}
Let $A$ be a simple Razak algebra. Then $A$ is approximately divisible and there is an isomorphism $\varphi: A \rightarrow A \otimes \mathcal{Z}$ such that $\varphi \sim_{a.u.} \id_A \otimes 1_\mathcal{Z}$.
\end{corollary}

\begin{proof}
For each $k \in \mathbb{N}$, let $\iota_k$ be a unital embedding of $M_k$ into $\mathcal{Q}$ and define a sequence of unital $^*$-homomorphisms $\mu_{k,m} : M_k \rightarrow  M(A\otimes\mathcal{Q}^{\otimes\infty})$ by
\begin{equation} \label{6.1}
\mu_{k,m} := 1_{\widetilde A} \otimes 1_\mathcal{Q} \otimes \cdots \otimes 1_\mathcal{Q} \otimes \underbrace{\iota_k}_m\otimes 1_\mathcal{Q} \otimes \cdots
\end{equation}
(where $\widetilde A$ is the minimal unitization of $A$). Then the sequence $(\mu_{k,m})_{m=1}^\infty$ is asymptotically central for $A\otimes\mathcal{Q}^{\otimes\infty}$. Thus $A \cong A\otimes\mathcal{Q}^{\otimes\infty}$ is approximately divisible. The second claim follows directly from Proposition \ref{uhf} since UHF algebras are $\mathcal{Z}$-stable, or alternatively from \cite[Theorem 2.3]{Toms:2005kq}, which says that separable approximately divisible $C^*$-algebras are $\mathcal{Z}$-stable. That the isomorphism $\varphi$ satisfies $\varphi \sim_{a.u.} \id_A \otimes 1_\mathcal{Z}$ is automatic (see \cite[Theorem 2.2]{Toms:2007uq}).
\end{proof}

Next, we combine the proofs of Corollary \ref{divcor} and \cite[Proposition 2.2]{Toms:2005kq} to construct an embedding of $W$ into the central sequences algebra $M(W)_\infty \cap W'$.

\begin{theorem} \label{Q-stable}
Let $B$ be a separable $\mathcal{Q}$-stable $C^*$-algebra. Then there exists a $^*$-homomorphism $\sigma = (\sigma_i)_{i=1}^\infty : W \rightarrow (\widetilde B \otimes \mathcal{Q}^{\otimes \infty})_\infty \cap (B \otimes \mathcal{Q}^{\otimes \infty})' \subset M(B \otimes \mathcal{Q}^{\otimes \infty})_\infty \cap (B \otimes \mathcal{Q}^{\otimes \infty})' \cong M(B)_\infty \cap B'$ which satisfies the nondegeneracy condition
\begin{equation} \label{nondegen1}
\|\sigma(a)b\| = \|a\|\|b\| \quad \text{for every $a\in W$ and $b\in B$}, 
\end{equation}
and which is trace-preserving in the sense that
\begin{equation} \label{nondegen2}
\lim_{i\rightarrow \infty} \rho(\sigma_i(a)) = \tau(a) \quad \text{for every $a\in W$ and every $\rho \in T^+_1(B)$}
\end{equation}
(where $\tau$ is the unique tracial state on $W$).
\end{theorem}

\begin{remark} \label{remark}
If $B$ is a $\mathcal{Q}$-stable $C^*$-algebra then every tracial state $\rho$ on $B \otimes \mathcal{Q}^{\otimes \infty}$
\begin{enumerate}[(i)]
\item is of the form $\tau \otimes \tau_\mathcal{Q}$ for some tracial state $\tau$ on $B$, where $\tau_\mathcal{Q}$ is the unique tracial state on $\mathcal{Q}$, and
\item extends uniquely to a tracial state on $\widetilde B \otimes \mathcal{Q}^{\otimes \infty} \subset M(B \otimes \mathcal{Q}^{\otimes \infty}) \cong M(B)$.
\end{enumerate}
This is what is meant in (\ref{nondegen2}). It may also be more natural to replace $B_\infty$ with an ultrapower $B_\omega$ for some free ultrafilter $\omega$, and the proof works equally well in this setting. 
\end{remark}

\begin{proof}
Write $W = \overline{\bigcup_{i=1}^\infty A_i}$ with the building blocks $A_i=A(n_i,(a_i+1)n_i)$ and inclusion maps $\varphi_{ij}$ given by Proposition \ref{sequence}. As in (\ref{6.1}), for each $k \in \mathbb{N}$ let $\iota_k$ be a unital embedding of $M_k$ into $\mathcal{Q}$, and define $^*$-homomorphisms $\mu_{k,m} : M_k \rightarrow \widetilde B \otimes \mathcal{Q}^{\otimes \infty} \subset M(B \otimes \mathcal{Q}^{\otimes \infty})$ by
\[
\mu_{k,m} := 1_{\widetilde B} \otimes 1_\mathcal{Q} \otimes \cdots \otimes 1_\mathcal{Q} \otimes \underbrace{\iota_k}_m\otimes 1_\mathcal{Q} \otimes \cdots.
\]
Note that the sequence $(\mu_{k,m})_{m=1}^\infty$ is asymptotically central for $B \otimes \mathcal{Q}^{\otimes \infty}$ and that $\|\mu_{k,m}(a)b\| \rightarrow \|a\|\|b\|$ as $m\rightarrow \infty$ for every $a\in M_k$ and $b\in B \otimes \mathcal{Q}^{\otimes \infty}$ (since this is true for finite tensors). Moreover, if $\rho$ is a tracial state on $B \otimes \mathcal{Q}^{\otimes \infty}$ and $\tr_k$ denotes the unique tracial state on $M_k$, then in the sense of Remark \ref{remark} we have $\rho(\mu_{k,m}(x)) = \tr_k(x)$ for every $x\in M_k$ and $m\in \mathbb{N}$.

For every $i \in \mathbb{N}$, let $\pi_i : A_i \rightarrow M_{n_i}$ be the irreducible representation $\ev_\infty$ (actually, any point evaluation will do), and define $\sigma_{i,m} := \mu_{n_i,m} \circ \pi_i : A_i \rightarrow \widetilde B \otimes \mathcal{Q}^{\otimes \infty}$. Let $(b_i)_{i=1}^\infty$ be dense in $B \otimes \mathcal{Q}^{\otimes \infty}$ and fix finite subsets $F_i \subset A_i$ such that $F_i \subset F_{i+1}$ and $\overline{\bigcup_{i=1}^\infty F_i} = W$. For each $i$, we can use the properties of the $^*$-homomorphisms $\mu_{n_i,m}$ to choose $m_i\ge m_{i-1}$ such that for $a \in F_i$ and $1 \le j \le i$ we have 
\begin{equation} \label{9}
\|[\sigma_{i,m_i}(a), b_j]\| \le 1/i
\end{equation}
and
\begin{equation} \label{10}
\bigg|\|\sigma_{i,m_i}(a)b_j\| - \|\pi_i(a)\|\|b_j\|\bigg| \le 1/i.
\end{equation}

Note also that, for every $a\in A_i$ and as $j\rightarrow \infty$ we have
\begin{equation} \label{11}
\|\pi_j \circ \varphi_{ij}(a)\| \rightarrow \|a\|
\end{equation}
by (\ref{small}) and (\ref{small3}) of Proposition \ref{sequence} and
\begin{equation} \label{15}
\tr_{n_j}(\pi_j \circ \varphi_{ij}(a)) \rightarrow \tau(a)
\end{equation}
by (\ref{6}) of Proposition \ref{simple}. Now we just patch together the $\sigma_{i,m_i}$ to get the desired map $\sigma$ (as in the proof of \cite[Proposition 2.2]{Toms:2005kq}). Since $\sigma_{i,m_i}$ is finite rank, by Arveson's extension theorem we can extend it to a linear, contractive (in fact c.c.p.) map $\overline{\sigma}_{i,m_i} : W \rightarrow \widetilde B \otimes \mathcal{Q}^{\otimes \infty}$. Define $\sigma$ to be the map $\sigma := (\overline{\sigma}_{i,m_i})_{i=1}^\infty : W \rightarrow (\widetilde B \otimes \mathcal{Q}^{\otimes \infty})_\infty$. Then $\sigma$ is linear, contractive and (since the $\sigma_{i,m_i}$ are $^*$-homomorphisms) is multiplicative on $\bigcup_{i=1}^\infty A_i$, hence on all of $W$. That is, $\sigma$ is a $^*$-homomorphism. By (\ref{9}), $\sigma(W)$ commutes with $B \otimes \mathcal{Q}^{\otimes \infty}$. For $a\in F_i$ and fixed $b_j$ we have
\[
\|\sigma(a)b_j\| = \limsup_{k} \|\sigma_{k,m_k}(\varphi_{ik}(a))b_j\| = \limsup_{k} \|\pi_k(\varphi_{ik}(a))\|\|b_j\| = \|a\|\|b_j\|
\]
by (\ref{10}) and (\ref{11}). Finally, for $a\in A_i$ and $\rho \in T^+_1(B \otimes \mathcal{Q}^{\otimes \infty})$ we have
\[
\lim_{k \rightarrow \infty} \rho(\sigma_{k,m_k}(a)) = \lim_{k \rightarrow \infty} \rho(\mu_{n_k,m_k} \circ \pi_k(\varphi_{ik}(a))) = \lim_{k \rightarrow \infty} \tr_{n_k}(\pi_k(\varphi_{ik}(a))) = \tau(a)
\]
by (\ref{15}). Therefore, $\sigma$ is a $^*$-homomorphism $W \rightarrow (\widetilde B \otimes \mathcal{Q}^{\otimes \infty})_\infty \cap (B \otimes \mathcal{Q}^{\otimes \infty})'$ that satisfies (\ref{nondegen1}) and (\ref{nondegen2}).
\end{proof}

It is easy to check that, by taking an approximate unit $(e_i)_{i=1}^\infty$ of $B$ and cutting $\sigma$ down by an appropriate subsequence of $(e_i\otimes 1_{\mathcal{Q}^{\otimes \infty}})_{i=1}^\infty$, we can get a trace-preserving completely positive contractive map from $W$ into $B_\infty \cap B'$ which preserves orthogonality (i.e.\ is `order zero'---see \cite{Winter:2009sf}).

\begin{corollary} \label{cent}
There exists a $^*$-homomorphism $\sigma = (\sigma_i)_{i=1}^\infty : W \rightarrow M(W)_\infty \cap W'$ which satisfies $\|\sigma(a)b\| = \|a\|\|b\|$ for every $a,b\in W$ and $\lim_{i\rightarrow \infty} \tau(\sigma_i(a)) = \tau(a)$ for every $a\in W$.
\end{corollary}

\end{document}